\documentclass[12pt]{article}
\usepackage{latexsym,amssymb,amsmath,amsfonts,amsthm,graphicx,url}
\numberwithin{equation}{section}
\newtheorem{theorem}{Theorem}[section]
\newtheorem{lemma}[theorem]{Lemma}

\newtheorem{conj}[theorem]{Conjecture}

\theoremstyle{definition}
\newtheorem{example}{Example}

\def\beq{ \begin{equation} }
\def\eeq{ \end{equation} }
\def\mn{\medskip\noindent}
\def\ms{\medskip}

\def\bn{\bigskip\noindent}

\def\fns{\footnotesize}

\def\square{\vcenter{\vbox{\hrule height .4pt
  \hbox{\vrule width .4pt height 5pt \kern 5pt
        \vrule width .4pt} \hrule height .4pt}}}
\def\eopt{\hfill$\square$}

\def\ep{\varepsilon}
\def\vep{\varepsilon}
\def\RR{\mathbb{R}}
\def\ZZ{\mathbb{Z}}
\def\NN{\mathbb{N}}

\def\hbr{\hfill\break}
\def\sqz{\kern -0.2em}

\def\clearp{}

\begin{document}

\title{Discontinuous phase transitions \\
in the one-dimensional \\
pair and triplet creation process}
\author{Rick Durrett \\
\small James B. Duke Emeritus Professor of Math,
 Durham, NC 27705}

\date{\today}						

\maketitle

\begin{abstract}
Dickman and Tom\'e (1991) introduced models on $\ZZ$ in which $k$ consecutive occupied sites give birth at rate $\lambda$, individual particles die at rate 1, and the configuration is subject to nearest neighbor stirring. They claimed that the pair creation model ($k=2$) has a continuous phase transition and was in the universality class of directed percolation, but the triplet creation model ($k=3$) has a discontinuous (or first-order) phase transition when the stirring rate is large. Hinrichsen (2000a) disputed the second claim and argued that first-order phase transitions in 1+1-dimensional nonequilibrium systems with fluctuating ordered phases are impossible. In this paper we prove rigorously that the phase transitions are discontinuous in both the pair and triplet models when the stirring rate is large. At the end of Section 2 we state some open problems and suggest reasons that the results in the physics literature differ from those presented here,
\end{abstract}

\section{Introduction}

The general setting for this investigation is the topic of ``nonequilibirum phase transitions.'' This term may be confusing because the models we will study have equilibria. In physics the term nonequilibirum refers to systems whose stationary distributions do not satisfy detailed balance. See for example Morro and Dickman (1999) or Hinrichsen (2000b). In Chapter IV of Liggett (1985) it is shown that in the finite range case, translation invariant spin systems with positive rates that are reversible with respect to some probability measure are stochastic Ising models with  interactions $J_R$ that depend on finite subsets $R$ of $\ZZ^d$. 

In the words of  Dickman and Tom\'e (1991),  we consider a population of particles on the sites of a one dimensional lattice with at most one particle per site, which evolves via spontaneous annihilation $X \to 0$, autocatalytic creation $kX \to (k+1)X$, and nearest neighbor hopping (known to probabilists as the simple exclusion or the stirring process). In continuous time when the configuration is $\xi \in \{0,1\}^\ZZ$, the death rate is $d(x,\xi)\equiv 1$ when $\xi(x)=1$ and the birth rate $b(x,\xi)$ is
\beq
(\lambda/2) \left[ 1_{(\xi(x+i)=1 \hbox{ \fns for all $1\le i  \le k$})} 
+ 1_{(\xi(x-i)=1 \hbox{ \fns for all $1\le i\le k$})} \right]
\label{catflip}
\eeq
when $\xi(x)=0$. By definition $b(x,\xi)=0$ when $\xi(x)=1$ and $d(x,\xi)= 0$ when $\xi(x)=0$.

These models are, of necessity, simulated on an interval of finite length, and when using a computer it is natural to work in discrete time.  Let $\sigma_j$ be the state of site $j$. On each step a site $i$ is chosen at random and then one of the following three elementary processes is performed.

\begin{itemize}

\item
{\it Hopping (or stirring).}  Probability $D<1$. The values at $\sigma_i$ and $\sigma_{i+1}$ are interchanged.

\item
{\it Creation.} Probability $(1-D) \lambda/(1+\lambda)$. If $\sigma_i= \ldots = \sigma_{i+k-1} = 1$ then one of the sites $i-1$ or $i+k$ is chosen at random, and if the site is vacant a birth occurs there.

\item
{\it Death.} Probability $(1-D)/(1+\lambda)$. If $\sigma_i=1$ then $\sigma_i$  is set equal to 0.

\end{itemize}

\noindent
These discrete time dynamics are equivalent to the continuous time process defined above in which stirring occurs at rate $\nu$, where 
$D = \nu/(\nu+\lambda+1)$ or $\nu = (\lambda+1)/(1-D)$.
In the work of Durrett and Neuhauser (1994) the processes takes place on $\ep\ZZ$ with stirring at rate $\ep^{-2}$ so $\ep = \sqrt{1-D}/\sqrt{\lambda+1}$. In this paper we will investigate the behavior of the model in continuous time when $\ep$ is small. To set the stage for those developments we will now describe some simulations from the physics literasture.

\mn
{\bf Dickman and Tom\'e (1991)}

\mn
examine the pair and triplet creation models in discrete time using the dynamics described above. Monte Carlo simulations were performed on a one-dimensional lattice with 10,000 sites and periodic boundary conditions.

\begin{figure}[h] 
  \centering
  \includegraphics[width=2.8in,keepaspectratio]{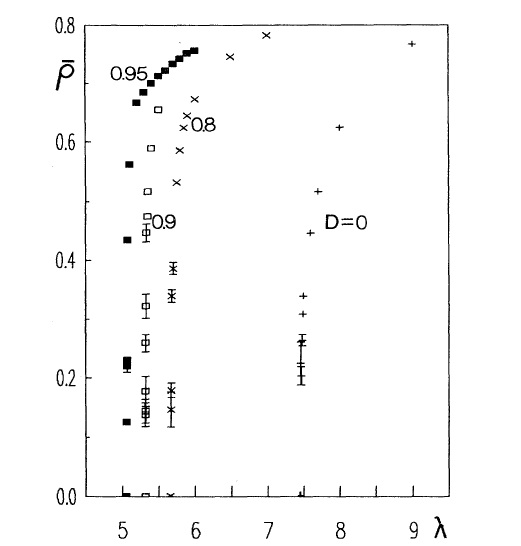}
 \includegraphics[width=2.8in,keepaspectratio]{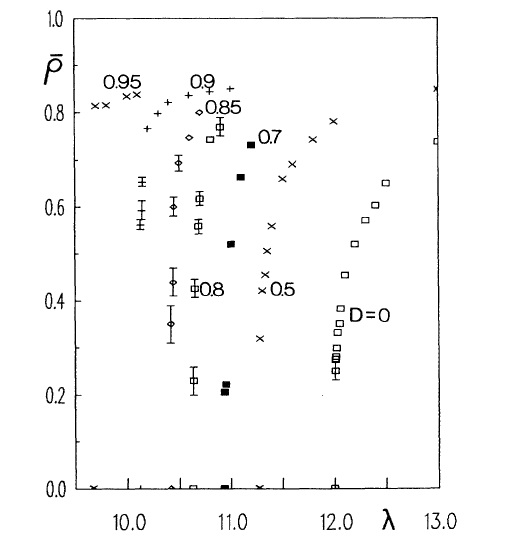}
\caption{Plot of average density vs.~creation rate in the pair model (left) and triplet creation.} 
\label{fig:denssim}
\end{figure}

\mn
{\bf Remarks.} Readers may say that there is a gap for the pair model, with no observations below 0.1. For models in the universality class of directed percolation the equilibrium density $\rho(\lambda) \approx (\lambda-\lambda_c)^{0.276}$ as $\lambda \downarrow \lambda_c$. If we forget about the fact that we don't know the constant then the density is 0.1 when $\lambda-\lambda_c=10^{-3.623} =2.38 \times 10^{-4}$.

In the right panel the curve for $D=0.95$ may look very short but if you look at the right panel in Figure \ref{fig:denspair} in Section 2.1 then you see that $\rho_0(\lambda_1) \approx 0.8$ and if the conjecture in \eqref{conjdisc} is true then the density never takes values in $(0,\rho_0(\lambda_1))$.

\bn
{\bf Hinrichsen (2000)}

\bn
The most famous part of his paper is the argument that discontinuous transitions are impossible in a class of one-diemnsional models that contains the triplet creation model. We refer the reader to his paper for the details. Here, we will be content to present the results of  a very interesting simulation of the model with $D=0.9$ and $\lambda = 10.145$ which is his estimate of the critical value when $D=0.9$. The following description is a combination of text in Section III and the caption  to Figure 4, to which we have added a few words of explanation.

\begin{figure}[h] 
  \centering
  \includegraphics[width=6.0in,keepaspectratio]{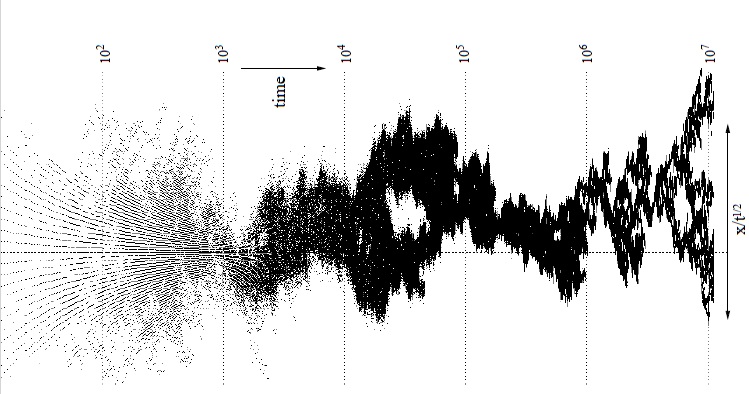}
\caption{The simulation starts with an interval of 20 occupied sites. Plot of occupied sites on a logarithmic times scale with space scaled by $x/t^{1/2}$.}
\label{fig:Hinsim}
\end{figure}

\medskip
Obviously there are three different temporal regimes. In the first 100 time steps, the 
island survives with certainty due to the large initial size of 20 sites, followed by a rapid decrease of the survival probability $p(t)$ of faster than $1/\sqrt{t}$. This part of the temporal evolution is supposed to be nonuniversal. From time $10^3$ to $10^5$, $p(t)$ decreases las $1/\sqrt{t}$ and the number of particles $n(t)$ stays almost constant. This is the time window where the model behaves essentially as the zero-temperature 
Glauber-Ising model so the transition appears to be discontinuous.

Up to time $10^4$, the cluster has the form of a compact diffusing clound of particles with very small islands of unoccupied sites generated by fluctuations, The first macroscopic minority island appears after $2 \times 10^4$ time steps indicating a crossover to directed percolation behavior which extends up to $10^6$ time steps. 
Only in the last deade from time $10^6$ to time $10^7$  where the thickness of active branches is small compared to the lateeral cluster size, do we observe the typical patterns of directed percolation.

\bn
{\bf Spreading simulations}

\mn
To see if the model is in the universality class of directed percolation, researchers choose the birth rate to be at the critical value, start with a fixed number of particles, and consider the following statistics:

\medskip
$p(t)$ the probability of survival to time $t$,

$n(t)$, the number of particles at time $t$.

\mn
At the critical point, these quantities follow the following asymptotic power laws
\beq
p(t) \sim t^{-\delta} \qquad n(t) \sim t^\eta 
\label{crexpsim}
\eeq
For directed percolation in 1+1 dimensions $\delta = 0.162$, and $\eta = 0.317$.
Figure 3 in Hinrichsen (2000a) shows plots of $p(t)t^{-\delta}$ and $n(t)t^{-\eta}$. If the triplet creation process is in the directed percolation universality class then the curves would converge to a constant. Perhaps due to the greatly reduced sample size for large $t$ the evidence for this behavior is not overwhelming, so we turn to Park (2009) for an illustration.

\begin{figure}[h] 
  \centering
  \includegraphics[width=4.0in,keepaspectratio]{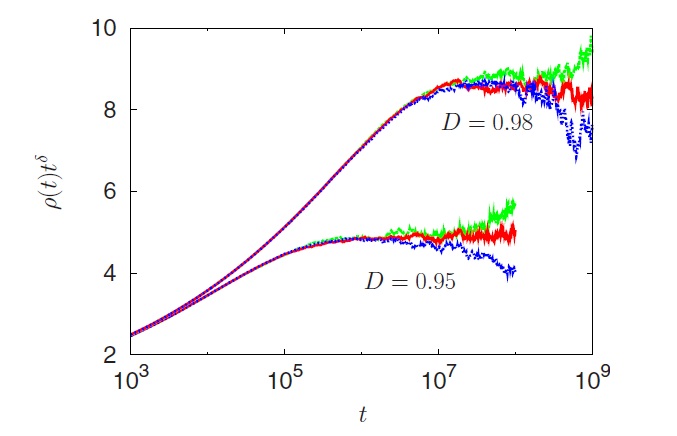}
\caption{Plots of $p(t) t^\delta$ verus $t$ for $\delta=0.1595$. The three curves for $D=0.95$ are for values close to $p=0.089895$ ($\lambda=10.124$), while for $D=0.98$,  values are close to $p=0.094226$  ($\lambda = 9.1628$)}
\label{fig:Parksim}
\end{figure}

\noindent
The model parameters used in Park's paper are somewhat different. Hopping occurs with probability $D'$, which means that the chosen particle attempts to jump and if the chosen target is occupied nothing happens.  As before, annihilation with probability $p(1-D')$ and births are attenpted with probability $(1-p)(1-D')$, where $p = 1/(1+\lambda)$.

\clearpage

\bn
{\bf Numerical estimates of critical values}

\begin{center}
\begin{tabular}{cccccccc}
$D$ & 0.7 & 0.8 & 0.85 & 0.9 & 0.9048 & 0.9608  & 1 \\
$\lambda_c(D)$ & 10.935 & 10.64 & 10.415 & 10.145 &10.124 & 9.613 \\
$\lambda_c(D)/\ep$ & 4.190 & 4.472 & 4.583 & 4.759 & 4.809 & 4.912 \\
fit & 4.187 & 4.471 & 4.614 & 4.756 & 4.769 & 4.929 & 5.040
\end{tabular}
\end{center}	

\noindent
The first three estimates of $\lambda_c(D)$ are from Dickman and Tom\'e (1991), the value for $D=0.9$ is from Hinrichsen (2000a). The last two come from Park (2009). The strange values of $D$ in the last two columns come from the fact, noted above, that they implement the diffusion step in a different way, so that what they call $D'=0.95$ corresponds to $D=D'/(2-D')=  0.9048$. As the next figure shows $\lambda_c(D)$ is fit well by a straight line.

\begin{figure}[h] 
  \centering
  \includegraphics[width=3.0in,keepaspectratio]{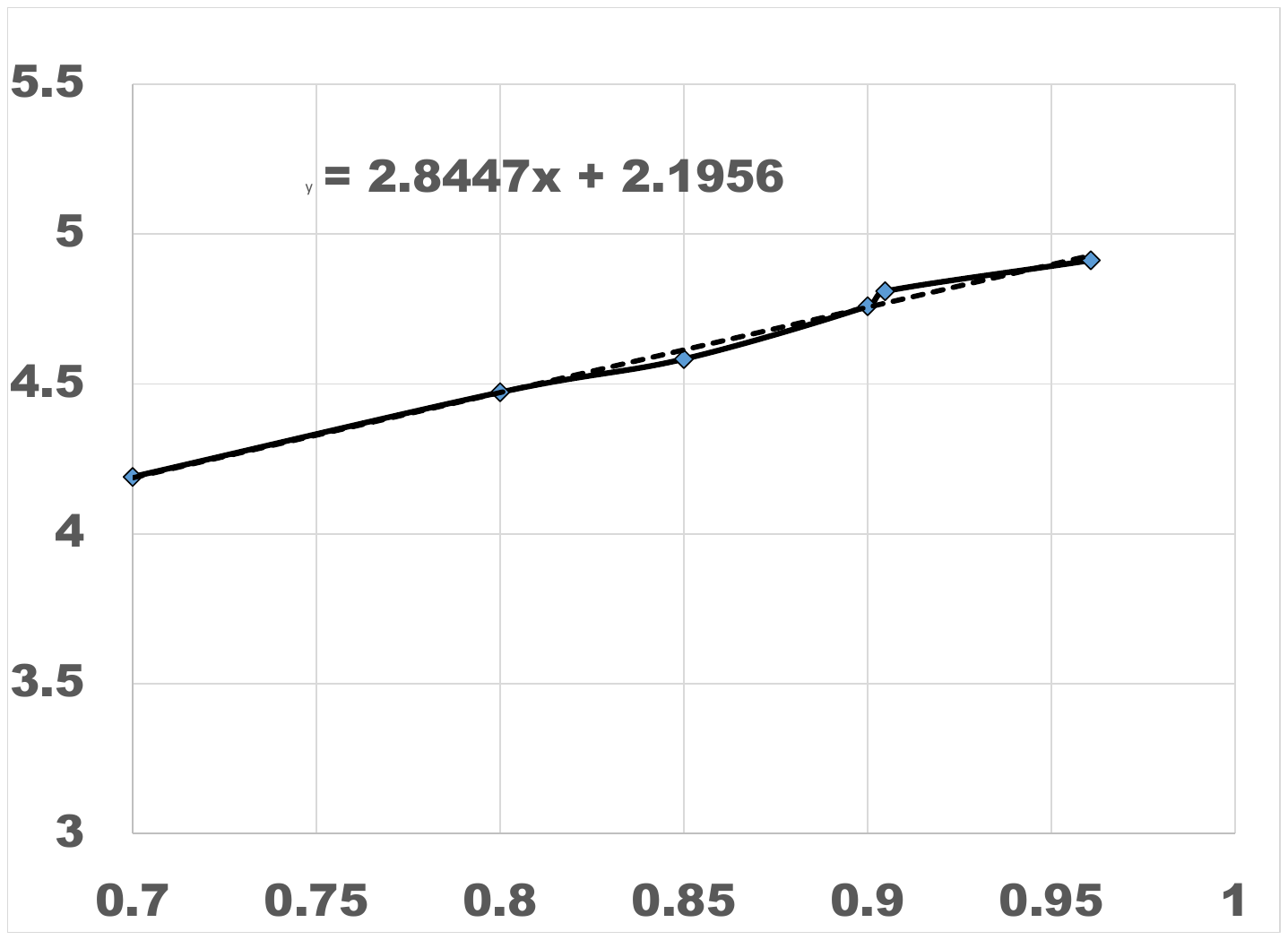}
\label{fig:denstrip}
\end{figure}

\mn
{\bf Other references}

\mn
There is a large literature on this topic,so I have only selected a few of the most important. Readers should also consult \'Odor (2003), Fiore and Oliveira (2004), Cardano and Fontanari (2006) and \'Odor and Dickman (2009) to have a more complete picture.

\clearp

\section{Results old and new}

\subsection {Durrett and Neuhuaser (1994)} \label{sec:DN94}

considered interacting particle systems on $\ep\ZZ^d$ subject to stirring at rate $\ep^{-2}$. Here we restrict our attention to systems in which each site can be in state 0 or 1, so the system is described by the flip rates $c(x,\xi)$ which is the rate $\xi(x)$ flips to $1-\xi(x)$ when the configuration is $\xi \in \{0,1\}^{\ep\ZZ}$. They assumed that

\mn
(i) $c(x,\xi)$ is translation invariant and has finite range:
$$
c(x,\xi) = h( \xi(x), \xi(x+\ep y_1), \ldots \xi(x+ \ep y_m) )
$$
(ii) For each nearest neighbors $x,y \in \ep \ZZ^d$ we exchange the values at $x$ and $y$ at rate $\ep^{-2}/2$.

\mn
Let $b(x,\xi) = c(x,\xi)1_{(\xi(x)=0)}$ be the birth rates and  
let $d(x,\xi) = c(x,\xi)1_{(\xi(x)=1)}$ be the death rates

In the limit of fast stirring neighboring sites become independent, so it should not be surprising that the rescaled particle system converges to the solution of the mean-field  partial differential equation, which is derived by assuming neighboring sites are independent.
We describe the proof of the next result in Section \ref{sec:PfTh1} because we will need to improve some of the estimates.

\begin{theorem} \label{DFLmf}
Suppose that $\xi^\ep_0(x)$ are independent and let $u^\ep(t,x) = P(\xi^\ep_t(x)=1)$. If $u^\ep(0,x) =v(x)$ is continuous, then as $\ep\to 0$, $u^\ep(t,x) \to u(t,x)$ the bounded solution of 
\beq
\frac{\partial u}{\partial t} = \frac{1}{2} \Delta u + f(u)
 \quad\hbox{with}\quad u(0,x)=v(x) ,
\label{rdpde}
\eeq
where the reaction term $f(u) = \langle b(0,\xi) \rangle_u - \langle d(0,\xi) \rangle_u$
and $\langle \ \cdot \ \rangle_u$ is expected value with respect to product measure with density $u$.
\end{theorem}

Two of the four examples in DN(1994) are the ones studied in this paper.

\begin{example} \label{exsexr}
{\bf Pair creation model.}  When $k=2$ the reaction term in Theorem \ref{DFLmf} is
$$
f(u) = -u + \lambda (1-u) u^2
$$
When $u \in [0,1]$, $u(1-u) \le 1/4$ so if $\lambda<4$ there is no root in $(0,1]$. If $
\lambda=4$ the point 1/2 is a double root, while if $\lambda>4$ there are two roots
$\rho_0 > \rho_1 = (1 \pm \sqrt{1-4/\lambda})/2$. We let $\lambda_0=4$ be the smallest value for which there is an equilibrium. If $\lambda> \lambda_0$ then 0 and $\rho_0$ are attracting fixed points while $\rho_1$ is unstable.
\end{example}

\begin{example} \label{excatr}
 {\bf Triplet creation model.} $k=3$. We begin by doing the calculations for general $k$. 
$f(u) = -u + (1-u) \lambda u^k$.
Factoring out $u$ we want to solve $\phi(u)=u^{k-1} - u^k = 1/\lambda$. 
$$
\phi'(u) = (k-1) u^{k-2} - k u^{k-1}
$$
so if we let $v_0= (k-1)/k$ then $\phi$ is increasing on $[0,v_0]$ and decreasing on $[v_0,1]$. If we let $\lambda_0 = 1/\phi(v_0)$ then $\phi(u) = 1/\lambda_0$ has a double root at $v_0$ while for $\lambda>\lambda_0$ the equation $\phi(u) = 1/\lambda$  has  exactly two roots $\rho_1 < \rho_0$ in $(0,1]$.  When $k=3$, $v_0=2/3$ and 
$$
\lambda_0 = \frac{1}{\phi(2/3) }= \frac{1}{ (2/3)^2 - (2/3)^3 } = \frac{27}{4} = 6.75
$$
If $\lambda> \lambda_0$ then 0 and $\rho_0$ are attracting fixed points while $\rho_1$ is unstable.
\end{example}

To proceed we need a little theory. See Liggett (1985) or (1999). The pair and triplet creation models are {\bf attractive}, which means that if we use the natural partial order on $\{0,1\}^{\ep\ZZ})$ where $\xi \le \eta$ when $\xi(y) \le \eta(y)$ for all $y$ then

\medskip
the birth rate $b(x,\xi)$ is an increasing function of $\xi$ when $\xi(x)=0$, 

the death rate $d(x,\xi)$ is a decreasing function of $\xi$ when $\xi(x)=1$. 

\mn
In the pair and triplet creation models $\xi(x) \equiv 0$ is an absorbing state. To try to construct a nontrivial stationary distribution we let $\xi^1_t$ be the process starting from $\xi^1_0(x)\equiv 1$. Since the systems are attractive $\xi^1_t$ converges to a limit  $\xi^1_\infty$ which is the largest possible stationary distribution. The limit might be $\delta_\emptyset$, the point mass on the all 0's configuration. In this case there is no nontrivial stationary distribution, so we define the critical value for the process to have a nontrivial stationary distribution to be 
$$
\lambda_c(\ep) = \min \{ \lambda: \xi^1_\infty \neq \delta_\emptyset \}.
$$

By combining known results for the limiting PDEs with a block construction,  DN(1994)  were able to prove results about the asymptotic behavior of the models as $\ep\to 0$. The key observation are, see page 293 of DN(1994).

\mn
(i) The relevant fixed point of the mean filed ODE which gives the equilibrium density in the spatial model with fast strirring is determined by the direction of motion of the traveling wave $u(t,x) = w(x-rt)$ connecting the two stable fixed points, $w(-\infty)=\rho_0$ and $w(\infty)=0$. 

\mn
(ii) The sign of the velocity $r$ is the same as the sign of the integral of $f$ from 0 to $\rho_0$. 

\medskip
This analysis leads to the following result

\begin{theorem} \label{crasy}
$\lambda_c(\ep) \to \lambda_1 = \inf\{ \lambda: r(\lambda) > 0 \}$. If $\lambda> \lambda_1$ then $P_\lambda( \xi^1_\infty = 1) \to \rho_0$ as $\ep \to 0$.
\end{theorem}

\mn
Since the details are important to the story, a sketch of the proof is given in 
Section \ref{sec:PfTh2}.

In the {\bf pair creation model} when $\lambda=4.5$ the corrresponding equilibira are
$$
\rho_0 > \rho_1 = \frac{1 \pm \sqrt{1- 8/9}}{2} = 2/3, 1/3
$$
Since $f(u)$ is a cubic, the fact that the roots are 2/3, 1/3, and 0 implies that the integral from 0 to $\rho_0$ is 0, and hence  $\lambda_1=4.5$. 

In the {\bf triplet creation model} we have to do more work. It is no harder to do the calculation for general $k$ so we do that. To find $\lambda_1$ we set
\beq
0 = \int_0^{\rho_1} (-x + \lambda x^k - \lambda x^{k+1}) \, dx 
= - \frac{\rho_0^2}{2} + \frac{\lambda \rho_0^{k+1}}{k+1}
- \frac{\lambda \rho_0^{k+2}}{k+2}
\label{acws1}
\eeq
Since $f(\rho_0)=0$ we have 
\beq
-\rho_0 + \lambda \rho_0^k - \lambda \rho_0^{k+1} = 0
\label{acws2}
\eeq 
Dividing \eqref{acws1} by $\rho_0/2$ and subtracting the result from \eqref{acws2} gives
$$
\lambda\left[\rho_0^k \left( 1 - \frac{2}{k+1} \right) 
- \rho_0^{k+1} \left( 1 - \frac{2}{k+2} \right) \right]= 0
$$
so we have $\rho_0 k/(k+2) = (k-1)/(k+1)$ and
$$
\rho_0 = \frac{(k-1)(k+2)}{k(k+1)} \qquad 
\lambda_1 = \frac{1}{\rho_0^{k-1}(1-\rho_0)}
$$
When $k=3$, $\rho_0 = 5/6$ and $\beta_1=6^3/5^2 = 8.64$.

\begin{figure}[h] 
  \centering
  \includegraphics[width=3.0in,keepaspectratio]{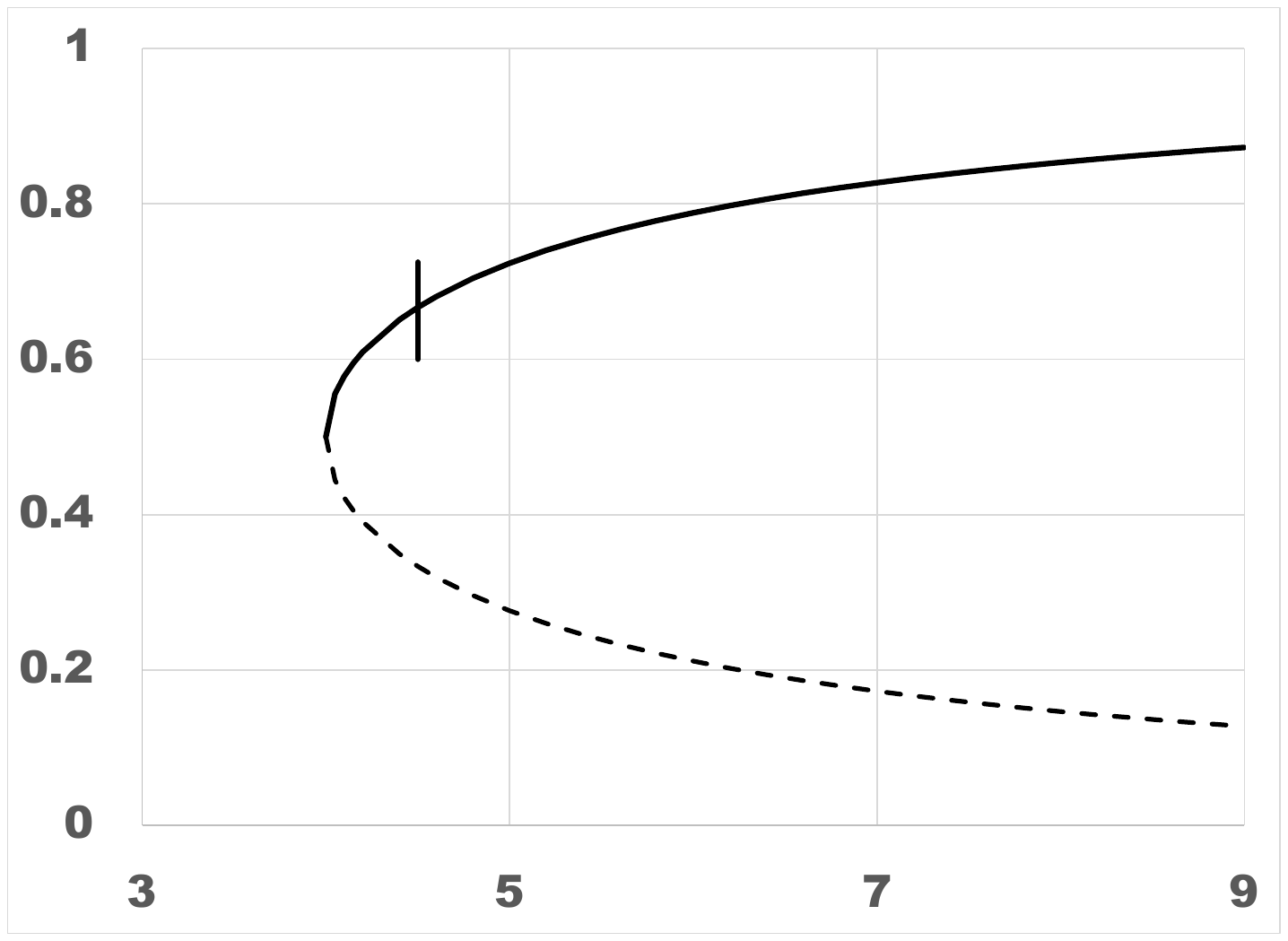}
  \includegraphics[width=3.0in,keepaspectratio]{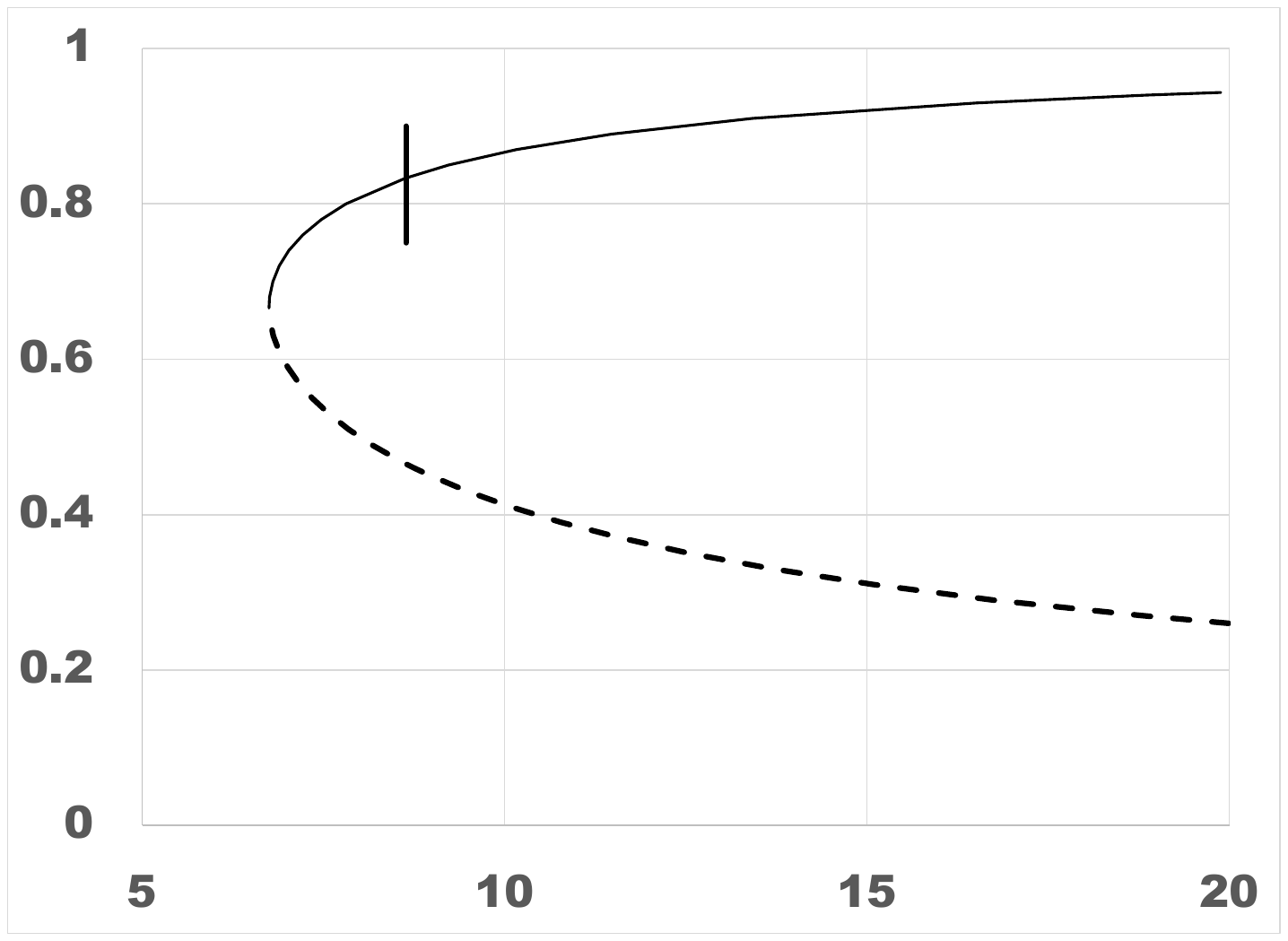}
\caption{ Left panel is pair creation model. Right panel is triplet creation model. Graph of $\rho_0(\lambda)$ (solid line) and $\rho_1(\lambda)$ (dashed line) . Vertical line marks the value at $\lambda_1$ which is the limiting critical value as $\ep\to 0$. } 
\label{fig:denspair}
\end{figure}

\subsection{Cox, Durrett, and Perkins (2013)} \label{sec:refres}

In the years since DN(1994) there have been improvements in the technology for proving results about the convergence of rescaled particle systems to solutions of PDEs. The ones we will discuss here can be found in CDP(2013). They consider voter model perturbations instead of fast stirring limits, but many of the details are the same, and the proofs for fast stirring are simpler. The dual of the voter model is coalescing random walks, so in that case we have to deal with loss of particles due to coalescence. In addition, to have a viable theory we need to restrict to dimensions $d\ge 3$ so that the voter model has a one parameter family of stationary distributions $\nu_p$, which are the substitute for the product measures $\mu_p$ for the stirring process. 

The first improvement in the hydrodynamic limit in Theorem \ref{DFLmf} is to replace the product measure initial conditions by a more general class. For simplicity we only state the definition in $d=1$. Given $r\in(0,1)$ we let $ \delta_\vep = \lceil \vep^{-r}\rceil \vep$, $Q_\vep = [0,\delta_\vep) \cap\vep\ZZ$, $|Q_\vep|$ be the number of points in $Q_\vep$,
and define the {\bf density near $x$ in the configuration $\xi$ }by
\beq \label{densdef0}
D^\ep(x,\xi) = \frac{1}{|Q_\vep|} \sum_{y\in Q_\vep}\xi(x+y)
\qquad\text{ for }x\in \delta_\vep\ZZ, \xi\in \{0,1\}^{\vep \ZZ}\,,
\eeq
 Given a continuous $v:\RR \to [0,1]$  we  say that a family of probability measures $\pi_\vep$ on $\{0,1\}^{\vep\ZZ}$ has {\bf density} $v$ if  for all $R,\delta>0$,
\beq\label{densdef}
\lim_{\vep\to 0}\sup_{x\in a_\vep \ZZ, |x|\le R}
\pi_\vep(|D^\ep(x,\xi)-v(x)|>\delta) = 0\,.
\eeq
The family of Bernoulli product measures $\bar\pi_\vep$
given by
\beq\label{ICconv}
\bar\pi_\vep(\xi(w_i)=1, i=1,\dots,n)
=\prod_{i=1}^n v(w_i) \text{ for all }n\in \NN \text{ and }
w_i\in\vep\ZZ .
\eeq
satisfies \eqref{densdef} for all  $r\in(0,1)$.

\begin{theorem}\label{conv}
 Let $x^k\in \RR$ and $x_\vep^k\in\vep\ZZ$, $\,k=1,\dots
  K$ satisfy
  \begin{equation}\label{xcond}
    x^k_\vep\to x^k\hbox{ and
    }\vep^{-1}|x^k_\vep-x^{k'}_\vep|\to\infty\hbox{ as
    }\vep\to0\hbox{ for any }k\neq k'. 
  \end{equation}
  If $u$ is the solution of mean-field PDE  \eqref{rdpde}, and $t>0$ then  
  \begin{align}
    \lim_{\vep\to 0}P(\xi^\vep_t(x_\vep^k+\vep y_i)=\eta_{i,k},
&\ \ i=0,\dots,L,\, k=1,\dots K)
\nonumber\\ 
& =\prod_{k=1}^K\langle 1\{\xi(y_i)=\eta_{i,k}, i=0,\dots,L\}\rangle_{u(t,x^k)}.
\label{localeq}
  \end{align}
for any  $\eta\in \{0,1\}^{\{0,\dots,L\}\times \{1,\dots K\}}$, $y_0,\dots,y_L\in\ZZ$ 
\end{theorem}

\mn
The last conclusion says that on a microscopic level the particle system is in local equilibrium which is product measure for stirring. In addition sites that are well separated on $\ep\ZZ$ in the limit are independent.

The second improvement replaces the previous conclusion about the asymptotics of the joint distribution of the $\xi^\ep_t(x)$ by one that shows that local space averages of the particle system are close to the solution of the PDE. We have modified the proof of Theorem 1.3 in CDP(2013) to compute 4th moments instead of variances to get a better result. The proof of our result is written for the triplet creation model but the result holds for finite range attractive spin systems subjected to stirring. However to prove things in that generality we would have to replace our construction of the process given in part a of Section  \ref{sec:PfTh1} by the more general one given at the beginning of Section 2 of DN(1994).

 Let $\delta_\ep = \lceil \ep^{-0.9} \rceil \ep$ so that the number of sites per interval
$n=\delta_\ep/\ep =  \lceil \ep^{-0.9} \rceil$. Let 
$m_\ep = \min\{ m: m \delta_\ep > \ep^{-0.3} \}$
and note that $m \le \ep^{-0.4}$.

\begin{theorem} \label{hydroLLN}
Let $v :\RR \to [0,1]$ be Lipschitz continuous, i.e., $|f(x)-f(y)| \le K|x-y|$ for all $x,y$. Suppose that the initial configuration has a density $v$ in the sense of \eqref{densdef}. If $t$ is fixed and $-m_\ep \le m < m_\ep$ then
$$
 P\left( \left| n^{-1} \sum_{y \in [m\delta_\ep, (m+1)\delta_\ep)} 
\xi_t(y) - u^\ep(t,y) \right| > \ep^{0.01} \right) \le  C\ep^{0.41}
$$
So with high probability this bad event will not occur for any $m \in [-m_\ep,m_\ep)$. 
\end{theorem}

\subsection{First-order phase transitions} \label{sec:firstor}

\mn
{\bf Biskup and Chayes (2003)} have developed general methods for proving that when mean-field theory predicts a discontinuous transition then if the dimension is sufficiently large or the the discontinuity in the mean-field system is sufficiently strong, the spatial model also has a discontinuous transition. A picture of the result for the three states Potts model is given in Figure \ref{fig:BCforbid}. Unfortunately their method cannot be applied to our models because they assume their systems are described by a Hamiltonian and hence have reversible stationary distributions given by Gibbs states.

\begin{figure}[h] 
  \centering
  \includegraphics[width=3.0in,keepaspectratio]{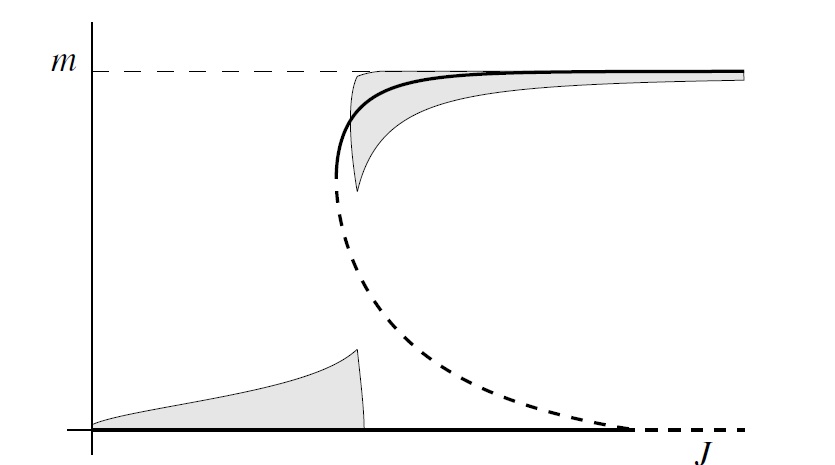}
\caption{The solutions of the mean-field eqautions for the Potts model. The solid line is the minimizer of the free energy while the dashed line gives the other solutions of the mean field equation.  The shaded regions show the set of allowed magnetizations for the system on $\ZZ^d$ when $I_d \le 0.002$. In addition to proving the existence of a discontinuity, these regions provide tight bounds on the transiition temperature and reasonable bounds on the size of the jump. } 
\label{fig:BCforbid}
\end{figure}

\clearp

\mn
{\bf The proof of discontinuity}

\mn
Let $\xi^{1,\ep}_\infty$ be the upper invariant measure, i.e., the limit starting from all 1's.  Let $p^\ep(\lambda) = P_\lambda (\xi^{1,\ep}_\infty(x)=1)$.  We will show that if $\ep$ is small then $\lambda \to p^\ep(\lambda)$ will drop discontinuously to 0 at $\lambda_c$. For simplicity, we will do the proof for the triplet creation model. The same proof works for the pair creation model. Following Chatterjee and Durrett (2013) we will prove the result by showing the existence of a {\bf forbidden region} ${\cal F}$ of pairs $(\lambda,p)$ which prevents the curve from reaching 0 continuously. Let $\lambda_2=15$ and define  ${\cal F} = {\cal F}_\ell \cup {\cal F}_0 \cup {\cal F}_r$ where
\begin{align*}
{\cal F}_\ell &= \{ \lambda < \lambda_1-\eta, p > 0 \} \\
{\cal F}_0 & = \{ \lambda \in [\lambda_1-\eta, \lambda_1+\eta], 
0 < p \le \rho_1(\lambda)-2\delta \} \\
{\cal F}_r & = \{ \lambda \in ( \lambda_1-\eta, \lambda_2],  
p \le\rho_0(\lambda) -2\delta \} 
\end{align*}

\begin{theorem} For any $\eta,\delta>0$ if $\ep$ is small then the curve, 
$(\lambda,p^\ep(\lambda))$ is never in $ {\cal F}$.
\end{theorem}

\begin{proof} Two thirds of this is very easy. Since $\lambda_c(\ep) \to \lambda_1$ as $\ep\to 0$, if $\ep$ is small then the only stationary distribution in ${\cal F}_\ell$ is $\delta_\emptyset$.
If $\lambda > \lambda_1$ then $p^\ep(\lambda) \to \rho_0(\lambda)$ so if $\ep$ is small then $p^\ep(\lambda) \ge \rho_0(\lambda)-\delta$ for $\lambda \in [\lambda_1+\eta,\lambda_2]$. 

Suppose that there is a sequence $\ep_n \to 0$ and $(\gamma_n,p_n) \in {\cal F}_0$
so that $P_{\gamma_n}( \xi^{1,\ep_n}_\infty(x)=1 ) = p_n$. Choose $t_n$ so that
 $P_{\gamma_n}( \xi^{1,\ep_n}_{t_n}(x)=1 ) = p'_n$ where $p'_n = \min\{p_n+\delta, 2p_n\}$. Since the initial configuration is all 1s,  which has density 1, it follows from Theorem \ref{hydroLLN} that $ \xi^{1,\ep_n}_{t_n}$ has density $p'_n< \rho_1(\gamma_n) \le \rho_1(\lambda)- \delta$. Note that the error bounds in that result only depends on $\ep$.  Using Theorem \ref{hydroLLN} again  it follows that if $T$ is large 
$\xi^{1,\ep_n}_{t_n+T}$ has density $\le p_n/2$. Since $\xi^{1,\ep_n}_{t_n+T}$ is stochastically larger than $\xi^{1,\ep_n}_{\infty}$ we have a contradiction which shows that the point $(\gamma_n,p_n)$ cannot exist for small $\ep$ which proves the desired result.
\end{proof}

Our argument suggests that as $\ep\to 0$
\beq
P_\lambda(x\in \xi^{1,\ep}_\infty) \to \begin{cases} 
\rho_0(\lambda)  & \lambda > \lambda_1 \\
0 & \lambda < \lambda_1 
\end{cases}
\label{conjdisc}
\eeq
but we cannot quite prove this, since we cannot rule out stationary distributions with densities with $\lambda\in (\lambda_1-\eta, \lambda_1+\eta)$ and $p \in [\rho_1(\lambda) - 2\delta, \rho_0(\lambda) - 2\delta]$ . 

As DN(1994) say on page 298, we 

\begin{conj} Associated with each point on the curve $(\lambda, \rho_0(\lambda))$, $\lambda \in(\lambda_0,\lambda_1)$ there is an associated metastable state. That is, if we start the system with parameter $\lambda \in(\lambda_0,\lambda_1)$ from an initial configuration of  all 1s then in the first $O(\log(1/\ep))$ of time the density of the system comes very close to $\rho_0(\lambda)$ and then persists in that density for a time $\ge \exp(c(\lambda)/\ep)$. 
\end{conj}

\mn
We further conjecture that the large deviations result of Kipnis, Olla, and Varadhan (1989) can be used to prove this. 

\subsection{Survival starting from a finite set}

In general an attractive interacting particle system with all 0s as an absorbing state has two critical values: 

\mn
$\lambda_e$ the smallest value of $\lambda$ for which there is a nontrivial stationary distribution,

\mn
 $\lambda_f$ the smallest value of $\lambda$ for which there is positive probability of survival from some finite set. 

\mn
The latter qualification is necessary because in the $k$-creation model the survival probability is 0 for sets of size $k-1$ or less. 

In the presence of fast stirring there is also the question of how large does the size of a set have to be (as a function of $\ep$) so that the survival probability does not go to 0 as $\ep \to 0$. It follows from the block construction (see part b in Section \ref{sec:PfTh2}) that there is a constant $C$ and a set $A_\ep$ with $|A_\ep| \le C/\ep$ so that if $\lambda > \lambda_1$ and we start from $A_\ep$ occupied then the probability the process survives does not go to 0 as $\ep\to 0$. To create $A_\ep$ we pick $L_\delta$ as in Theorem \ref{PDEc2posr} and let $A_\ep$ be a subset of $[-L_\delta,L_\delta] \cap \ep \ZZ$ with density $p \ge \rho_1(\lambda)+\delta$ in the sense of \eqref{densdef}. If we adopt the liberal definition that one can use sets of any size then 
$$
\limsup_{\ep\to 0} \lambda_f(\ep) \le \lambda_1 = \lim_{\ep\to 0} \lambda_e(\ep).
$$
If we use the strict definition and restrict our attention to sets of a fixed size then $\lambda_f(\ep) \to\infty$, because the probability of having a birth goes to 0. Finding the critical size of an interval that has an asymptotic positive probability of survival seems a difficult problem.

One of the reasons for raising the issue of survival from finite sets is that it can provide a possible explanation of why some conclusions in the physics literature are different from the ones proved here. If we start the triplet (or pair) creation process at the critical value $\lambda_f(\ep)$ from an interval of fixed size it will die out, and the critical values $\delta$ and $\eta$ could be the same as (or at least close to) those of oriented percolation. However the flaw in this argument for the triplet creation process to belong to th universalaity class of orineted percolation is that the process dying out at the critical value does not preculde the existence of a nontrivial stationary distributions.

As noted earlier, the simulations of Dickman and Tom\'e for the triple creation model with $D=0.95$ shown in Figure \ref{fig:denssim} agree with the asymptotic results of DN(1994) given in Figure \ref{fig:denspair}. In the pair model when two particles that are initially adjacent are moved by stirring they will return infinitely often to be adjacent and the amount of time they are adjacent up to time $t$ is of order $C\ep t^{1/2}$ see \eqref{Ewbd}. 

In the triplet model if we start with three particles $X_{-1}(t)$, $X_0(t)$, and $X_1(t)$  at $-\ep$, 0, and $\ep$ then the relative positions $(X_{-1}(t)-X_0(t), X_{1}-X_0(t))$ are a two dimensional random walk, so the time they are together up to time $t$ is of order
$\ep^2 \log(t/\ep^2)$. From this argument we see that if we start with four particles, their differences from the location of one of the particles are a three dimensional random walk which is transient so the expected amount of time they are all adjacent is $\le C\ep^2$. These observations and the estimates in part b of Section \ref{sec:PfTh1} suggest that the convergence to the mean-field limit is much slower in the pair creation model and might mean that observing the discontinuous behavior would require values of $D$ much closer to 1 than $0.95$. the mathematical argument only proves that it occurs when $ep$ is sufficently small.

\section{Proof of Theorem \ref{DFLmf}} \label{sec:PfTh1}

This proof is adapted from Section 2 of DN(1994) and subdivisions a-e parallel those parts of Section 2 of that paper. We reduce the generality of that proof to only apply to the $k$ particle creation procesess in $d=1$. In general, the argument is very close to that in DN(1994). In those parts of the proof that do not change we do not give many details. Howeve,r at the end of part b, we will work to improve the estimate of the probability of collisions in order to get a better assessment of the error in the approximation, and in subsection f we indicate how the argument must be modified to deal with an initial configuration with a density in the sense of \eqref{densdef}. 

\mn
{\bf a. Construction and time reversal}

\mn
The first step is to construct the process from a collection of Poisson processes, all of which are assumed to be independent. This structure is a close relative of the {\bf graphical representation} used to construct additive processes, see e.g., Griffeath (1979),  but the construction in DN(1994) applies to all attractive processes, so we will call  this a {\bf Poisson process construction.} The version we use here is a special construction for the $k$-particle  creation model

\begin{itemize}

\item
For $x,y \in \ep\ZZ$ with $|x-y|=\ep$, let $\{S^{x,y}_n, n \ge 1\}$ be Poisson processes with rate $\ep^{-2}/2$  At the arrival times of these Poisson process we draw an arrow from $x\to y$ and another from $y \to x$ to indicate that a stirring occurs.

\item
For $x \in \ep\ZZ$, let $\{ D^x_n, n \ge 1\}$ be Poisson processes with rate $1$.  At the arrival times of these Poisson processes we write a $\bullet$ at $x$ that kills any partilce on the site.

\item
For $x \in \ep\ZZ$, let $\{B^{x,+}_n, n \ge 1 \}$ and let $\{B^{x,-}_n, n \ge 1 \}$ be Poisson processes with rate $\lambda/2$. At the arrival times of $B^{x,+}_n$ we draw an arrow from $x+k$ to $x$ to indicatee that there will be a birth at $x$ if $x+1. \ldots x+k$ are all occupied and $x$ is vacant. At the arrival times of $B^{x,-}_n$ we draw an arrow from $x-k$ to $x$ to indicate that there will be a birth at $x$ if $x-1. \ldots x-k$ are all occupied and $x$ is vacant.

\end{itemize}

We now define the {\bf influence set}, $I^{x,t}(s)$, $s\le t$, which gives the set of sites at time $t-s$ that we need to know the states of in order to compute the state of $x$ at time $t$. Here when we wrrte the notation for the Poisson process without the subscript $n$ we are referring to all the points in it, and if we write $B^x$ we mean $B^{x,-} \cup B^{x,+}$. 

\begin{itemize}

\item
If $y \in I^{x,t}_\ep(s)$ and $t-s \in S^{y,z}$ or  $t-s \in S^{z,y}$ we move the particle from $y$ to $z$

\item
If $y \in I^{x,t}_\ep(s)$ and $t-s \in D^{y}$ then we remove $y$ from $ I^{x,t}_\ep(s)$.

\item
If $y \in I^{x,t}_\ep(s)$ and $t-s \in B^{y,+}$ then we add particles at $y+1, \ldots y+k$ to  $I^{x,t}_\ep(s)$. If $t-s \in B^{y,-}$ then we add particles at $y-1, \ldots y-k$ to  $I^{x,t}_\ep(s)$. 
\end{itemize}

\mn
If the third transition adds a particle to a site that is already occupied in $I^{x,t}_\ep(s)$ we say that a {\bf collision} has occurred and we call the new particle {\bf fictitious}. Fictitious particles are ignored in the dual, but for the proof it is convenient to allow them to move and to give birth to other particles (which will also be fictitious) so we give each of them their own independent copy of the Poisson process construction. Death events and stirrings cannot cause collisions.

\begin{figure}[ht]
\begin{center}
\begin{picture}(320,220)
\put(30,30){\line(1,0){260}}
\put(30,210){\line(1,0){260}}
\put(40,30){\line(0,1){180}}
\put(70,30){\line(0,1){180}}
\put(100,30){\line(0,1){180}}
\put(130,30){\line(0,1){180}}
\put(160,30){\line(0,1){180}}
\put(190,30){\line(0,1){180}}
\put(220,30){\line(0,1){180}}
\put(250,30){\line(0,1){180}}
\put(280,30){\line(0,1){180}}
\put(20,27){0}
\put(20,207){$t$}
\put(220,175){\vector(-1,0){60}}
\put(100,135){\vector(1,0){60}}
\put(280,115){\vector(-1,0){60}}
\put(40,90){\vector(1,0){60}}
\put(190,70){\vector(-1,0){60}}
\put(157,110){$\bullet$}
\put(97,55){$\bullet$}
\put(217,85){$\bullet$}
\linethickness{1.0mm}
\put(160,210){\line(0,-1){100}} 
\put(40,90){\line(0,-1){60}}
\put(70,90){\line(0,-1){60}}
\put(100,135){\line(0,-1){80}}
\put(130,135){\line(0,-1){105}}
\put(190,175){\line(0,-1){145}}
\put(220,175){\line(0,-1){90}}
\put(250,115){\line(0,-1){85}}
\put(280,115){\line(0,-1){85}}
\put(160,70){\line(0,-1){40}} 
\put(37,18){1}
\put(67,18){2}
\put(97,18){3}
\put(127,18){4}
\put(157,18){5}
\put(187,18){6}
\put(217,18){7}
\put(247,18){8}
\put(277,18){9}
\end{picture}
\caption{Influence set $I^{5,t}_s$ for the pair creation model with no stirring. It is easy to see that 5 is occupied at time $t$ if and only if all th sites in $\{1,2,4\}$, $\{1,2,5,6\}$ or $\{6,8,9\}$ are occupied. }
\end{center}
\end{figure}
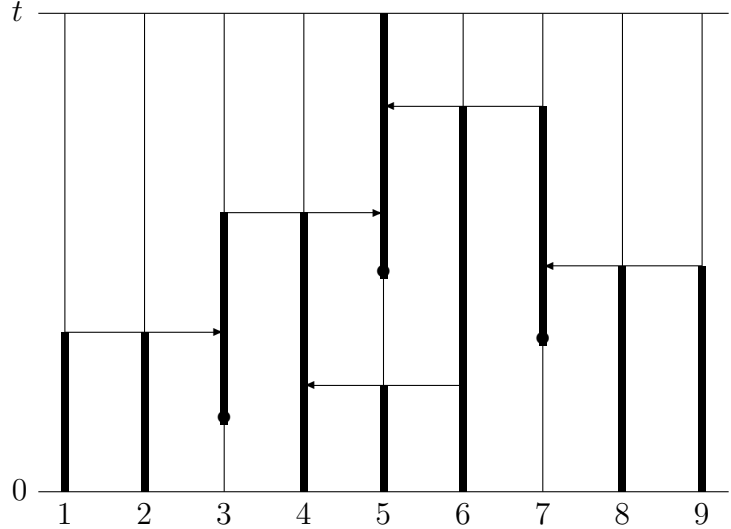

\bn
{\bf b. Particle motions are almost independent random walks} 

\mn
Thanks to our use of fictitious particles the total number of particles in $I^{x,t}_\ep(s)$ is a branching process that we call $Z_s$. To avoid the headache that this and other processes derived from the influence set  $I^{x,t}_\ep(s)$ are only defined for $s \le t$, we extend the Poisson processes to be defined for negative times.
In $Z_s$, particles die at rate 1 and give birth to 3 new particles at rate $\lambda$. To define the locations of the particles we let $X^0(0) =x$ and at the time of the $m$th birth, $\beta_m$, we add 3 new particles numbered $X^{3(m-1)+1}$, $X^{3(m-1)+2}$, $X^{3m}$ which start at time $\beta_m$ at their birth locations. 
We say that $X^k$ is {\bf crowded} at time $s$ if $|X^k-X^j| =\ep$ for some $j\neq k$. For bookkeeping purposes, unborn particles are assigned location $\infty$. To get an upper bound on the deviation from independent random walks we assume that the particles keep walking after they are dead. We will define a family of random walks $Y^k$, whose increments are independent, independent of the increments of the other random walks, are coupled to be close to the $X^k$. 

When $X^k$ is not crowded its increments are the same as $Y^k$. When $X^k$ is crowded we let $Y^k$ move independently of $X^k$. To estimate the amount of time that $X^k$ is crowded, let  $V(s) = X^k(s)-X^j(s)$, $j \neq k$, and let $W(s)$ be a random walk that jumps by $\ep$ or $-\ep$ at rate $2\ep^{-2}$. If $V(s)=x$ with $|x|=\ep$ and $y$ is another point $\neq -x$ and with $|y|=\ep$ (since we are in $d=1$ the only option is $y=x$)  then 

\begin{center} 
\begin{tabular}{ccc}
jumps $x$ to & rate in $V$ & rate in $W$ \\
$-x$ & $\ep^{-2}/2$ & 0 \\
0 & 0 & $\ep^{-2}$ \\
$x+y$ & $\ep^{-2}$ &  $\ep^{-2}$
\end{tabular}
\end{center}

\noindent
From this it should be clear that if we cut the visits to 0 out of $W(s)$ and call the result $\widehat W(s)$ then $|\widehat W(s)|$ and $|V(s)|$ have the same distribution. It follows that for any integer $M\ge 1$
$$
|\{s \le t: |V_s| \le M\ep \}| = v^{M\ep}_t \le_d
 w^{M\ep}_t = |\{s \le t: |W_s| \le M\ep \}| 
$$
where $\le_d$ is short for is smaller in distribution, that is the two random variables can be constructed on the same space so that $v^{M\ep}_t \le w^{M\ep}_t$.

Well known asymptotics for random walks imply that if $t\ep^{-2} \ge 2$ then
\beq
Ew^{M\ep} \le CM\ep t^{1/2} .
\label{Ewbd}
\eeq
Let $\chi^k_\ep(t)$ be the amount of time $X^k$ is crowded in $[0,t]$ then
$$
E(\chi^k_\ep(t)|Z_t=K) \le KEw^\ep_t .
$$
Since $EZ_t \le e^{3\lambda t}$ we have
\beq
E(\chi^k_\ep(t)) \le  e^{3\lambda t} Ew^\ep_t .
\label{Echit}
\eeq

To estimate the difference between $X^k(t)$ and $Y^k(t)$ we observe that if $\chi^k(t) = \tau$ then the independent jumps that occur in $[0,t]$ have a Poisson distribution with mean $\ep^{-2}\tau$. Let $\Delta_Y(t)$ be the net effect of the independent jumps up to time $t$. DN(1994) use second moment methods. Here we will instead use the Azuma-Hoeffding inequality. $\Delta_Y(t)$ is martingale with jumps bounded by $2\ep$ so if there are $m$ jumps
$$
P\left(  \max_{0\le s \le t} |\Delta_Y(s)| \ge x \right) \le \exp(-x^2/2m\ep^2) 
\le \exp(-x^2/2\ep^{0.95})
$$
Here to have an upper bound on the number of jumps up to time $t$ which fails with exponentially small probability even when $t = \log(1/\ep)$ we let $m=\ep^{-1.05}$.

The arguments leading to the last estimate also apply to $\Delta_X(t)$ the net effect of jumps of $X^k$ while it is crowded so we have the following bound on the distance between the original random walks produced by stirring and the ones with independent increments
\beq
P\left( \max_{0\le s \le t} |X^k(s)-Y^k(s)| \ge 2\ep^{0.45} \right) 
\le 2 \exp(-\ep^{-0.05}/2)) \quad\hbox{for}\quad t\le \log(1/\ep)
\label{stirvind}
\eeq

The last estimate shows that the individual particle motions are close to independent random walks. To show that with high probability no collisions occur (births onto occupied sites in the dual) we now multiply this bound in \eqref{Ewbd} on the expected amount of time two particles are close by the expected number of pairs of particles which is
$$
\le EZ_t^2 \sim C \exp(6t)
$$
since $W_t = Z_t/\exp(3\lambda t)$ converges a.s. to a limit $W$ and there is enough domination for us to conclude that $EW_t^2 \to EW^2$. Combining this with \eqref{Ewbd} with $M=3$ we see that if ${\cal C}^\ep_t$ is the number of collisions up to time $t$ then
\beq
E{\cal C}^\ep_t \le \ep C t^{1/2} \exp(6\lambda t).
\label{collbd}
\eeq
This estimate extends easily to bound the probability of collisions between two duals, we just compute for a baranching process starting with two particles.

\bn
{\bf c. Weak convergence of particle motions to Brownian motions}

\mn
Since the particles are moved by stirring, their positions are continuous time random walks $S_t$. Using standard arguments to compare random walks with Brownian motion it is snown, see (2.18) in the paper that for $t \ge \ep^{0.8}$
\beq
P(|S_s-B_s| \ge 2\ep^{0.32} \hbox{ for some $s\le t$}) \le \ep^{0.4} t
\label{StvsBt}
\eeq

\bn
{\bf d. Convergence of $u^\ep(t,x)$}

\mn
Combining arguments from parts b and c shows that the dual converges to a limit. We will have to work a little harder when we suppose that the initial configuration has a density (see subsection f), but at this point we are assuming that the initial condition is a product measure. 

\ms
{\bf To determine the state of $x$ at time $t$} we use independent uniform random variables to determine the states of the sites $y$ in $I^{x,t}_t$ which are occupied with probability $v(y)$ and then work our way up the Poisson construction. When we encounter a birth arrow from say $y+3 \to y$, if $y$ is occupied in the dual at that point we can ignore the birth. If it is not and all the particles at $y+3,y+2,y+1$ areoccupied  in the dual (one or more may have died in the evolution or not been in the initial configuration) we declare $y$ to be occupied. In either case, the three particles will disappear from the influence set after we go past the arrow. When we reach time 0 in $I^{x,t}_s$ the only point left is $x$ and we know whether $x$ was occupied at time $t$. 

\ms
{\bf To determine the probability $x$ is occupied at time $t$}, we replace $I^{x,t}_s$ by $\pi^{x,t}_s$ which gives the probabilities that the sites $y$ in  $I^{x,t}_s$ are occupied at time $t-s$. To get the calculation started we let $\pi^{x,t}_t(y) = v(y)$ be the probability that the particles in the dual at time $0$ is occupied. When we encounter a birth arrow we update the current value of $\pi(y)$ to be $\pi(y) + (1-\pi(y)\pi(y+1)\pi(y+2)\pi(y+3)$. Using this algorithm when we reach time 0 in the dual (time $t$ in the process going forward we will have computed $u(t,x)$. The paper by Huang and Durrett (2021) has a much simpler approach to the duality for the pair creation model.

\bn
{\bf e. The limit satisfies the PDE.} 

\mn
At this point we have shown that in the limit neighboring sites are independent. This is the assumption used to derive the mean field PDE so it should be clear that it holds. 
There are, of course, some technical problems invovled in interchaning limits and derivatives. We refer the reader to Section 2e of DN(1994) or Section 3.1 of CDP(2013) for details.

\bn
{\bf f. Initial configurations with a density}

\mn
Suppose that we divide $\ep\ZZ$ into intervals of length $a_\ep$ whose length is a multiple of $\ep$, for example $a_\ep=\lceil \ep^{-r} \rceil \ep$. Let $b_\ep = a_\ep^{1.9}$. In Section 3 of DN(1994) they construct stationary distributions using a block construction in which happy intervals have a density of 1s larger than $\rho_1(\lambda)+\delta$. See Lemma 3.3 in DN(1994) or Section \ref{sec:PfTh1} for more details. The key observation is (a) on  page 309 which we now state:

\begin{lemma} Starting from a fixed configuration with density $p$ at time 0 is almost the same as starting from a product measure with density $p$  at time $b_\ep$.
\end{lemma}

\mn
{\it Sketch of proof.} The first step is to observe that if we are working back from $x$ at time $t$ then as $\ep\to 0$ the probability that a branching occurs in $[0,b_\ep]$ tends to 0 as $\ep \to 0$. A particle at $y$ at time $b_\ep$ will trace its way back to a site in the initial configuration at a point that is located at $y + b_\ep^{1/2} \chi$ where $\chi$ has a normal distribution. Suppose without loss of generality that $y=0$. The distance $b_\ep^{1/2} = a_\ep^{0.95} \gg a_\ep$, so unless we are very far out in the tail of the normal distribution then when we condition on the particle ending up ia particular interval $[ka_\ep,(k+1)a_\ep]$ its location will be almost uniform over the interval. Calculations in DN(1994) show that the distribution is uniform if $|k| < a_\ep^{-0.05}$ and that displacemnts of this size are extremely unlikely. \eopt

\clearp

\section{Proof of Theorem \ref{crasy}} \label{sec:PfTh2} 

\bn
{\bf a. PDE results}

\mn
In the pair and triplet creation models when $\lambda > \lambda_0$ the reaction terms $f(u)$ have three roots $0 < \rho_1 < \rho_0$ with $0,\rho_0$ stable and $\rho_1$ unstable. We need two different PDE results depending on the sign of the speed $r$ of the traveling wave. Proofs can be found in Fife and McLeod (1977).

\begin{theorem} \label{PDEc2posr}
Suppose $r>0$, $u(0,x) \in [0,1]$ for all $x$ and $u(0,x) \ge \rho_1+\delta$ when $|x| \le L$. If $L \ge L_\delta$ and $\eta>0$ there are constants $0 < c,C < \infty$ so that
$$
u(t,x) \ge \rho_0 - C e^{-ct} \qquad\hbox{for $|x| \le (r-\eta) t$}.
$$
\end{theorem}

\noindent
We need the solution to be larger than $\rho_1+\delta$ on a large enough interval or the action of the Laplacian will bring the maximum of the solution to $< \rho_0$ and $u(t,x)$ will be doomed to converge to 0.

\begin{theorem} \label{PDEc2negr}
Suppose $r<0$, $u(0,x) \in [0,1]$ for all $x$, and $u(0,x) \le \rho_1-\delta$ when $|x| \le L$. If $L \ge L_\delta$ and $\eta>0$ there are constants $0 < c,C < \infty$ so that
$$
u(t,x) \le C e^{-ct} \qquad\hbox{for $|x| \le (-r-\eta) t$}.
$$
\end{theorem}

\noindent
In words, a large enough low density interval will grow linearly even if all sites outside are occupied.

\mn
{\bf b. Block construction} 

\mn
A ten lecture course on this technique can be found in Durrett (1995).
The idea is to  compare the behavior of the particle system in space-time with oriented site percolation on
$$
{\cal L} = \{ (x,n) : x + n \hbox{ is even}, n \ge 0 \}.
$$
The points $(m,n) \in {\cal L}$ are mapped to $v_{m,n}=(2Lm,Tn)$ in space-time.  Associated with each point $v_{m,n}$ is a spatial interval 
$$
I_{m,n} = (2Lm + [-L,L]) \times \{nT\} .
$$ 
There is a set $H$ of {\bf happy configurations} for the process in the intervals $I_{m,n}$, so that if the event occurs $(m,n)$ will be an open site in the percolation process. When we are constructing nontrivial stationary distributions for our models the happy event will be that the configuration in $I_{m,n}$ has density at least $\rho_1+\delta$ in the interval in the sense of \eqref{densdef}

Let ${\cal R}_{0,0} = [-4L,4L] \times [0,T]$
be the space time box associated with $(0,0)$, and let
${\cal R}_{m,n} = v_{m,n} + {\cal R}_{0,0}$. We will suppose that our process is defined using a family of Poisson processes as described at the beginning of Section \ref{sec:PfTh1}. Our goal will be to find a {\bf good event} $G_{m,n}$ measurable with respect to the Poisson points in ${\cal R}_{m,n}$ and having probability close to 1 so that if $I_{m,n}$ is happy and $G_{m,n}$ occurs then with high probability $I_{m-1,n+1}$ and $I_{m+1,n+1}$ are happy.

\begin{figure}[ht]
\begin{center}
\begin{picture}(260,210)
\put(0,30){\line(1,0){260}}
\put(30,180){\line(1,0){200}}
\put(30,30){\line(0,1){150}}
\put(230,30){\line(0,1){150}} 
\put(240,175){$T$}
\put(120,60){\vector(-1,3){30}}
\put(140,60){\vector(1,3){30}}
\put(118,135){${\cal R}_{0,0}$}
\put(5,20){\line(0,1){20}}
\put(55,20){\line(0,1){20}}
\put(105,20){\line(0,1){20}}
\put(155,20){\line(0,1){20}}
\put(205,20){\line(0,1){20}}
\put(255,20){\line(0,1){20}}
\put(55,170){\line(0,1){20}}
\put(105,170){\line(0,1){20}}
\put(155,170){\line(0,1){20}}
\put(205,170){\line(0,1){20}}
\put(40,195){$-3L$}
\put(95,195){$-L$}
\put(150,195){$L$}
\put(195,195){$3L$}
\put(95,45){$-L$}
\put(150,45){$L$}
\put(70,160){$I_{-1,1}$}
\put(170,160){$I_{1,1}$}
\put(120,10){$I_{0,0}$}
\put(20,10){$I_{-2,0}$}
\put(220,10){$I_{0,2}$}
\thicklines
\linethickness{1mm}
\put(5,30){\line(1,0){50}}
\put(105,30){\line(1,0){50}}
\put(205,30){\line(1,0){50}}
\put(55,180){\line(1,0){50}}
\put(155,180){\line(1,0){50}}
\end{picture}
\caption{\small Block construction.}
\end{center}
\label{fig:blockcon}
\end{figure}
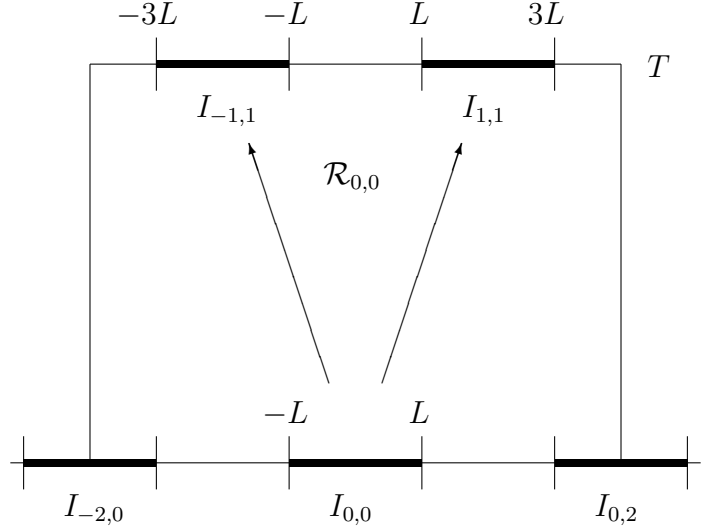

When $G_{m,n}$ occurs we make the site $(m,n)$ open in  the percolation process on ${\cal L}$, otherwise it is closed. Let $W_0$ be the set of even integers $x$ so that $I_{x,0}$ is happy. These sites are said to be {\bf wet} in the percolation process. Define the {\bf wet region at time} $n$ by
$$
W_n = \{ y : (x,0) \to (y,n) \hbox{ for some $x \in W_0$} \}
$$
where $(x,0) \to (y,n)$ means there is a sequence $x_0=0, x_1, \ldots x_n=y$ so that for $0 \le i <  n$ the sites $(x_i,i)$ are open and $|x_i-x_{i-1}|=1$. To see the reason for this convention that does not require $i=n$ to be open note that in Figure \ref{fig:blockcon} if $(0,0)$ is wet and $G_{0,0}$ occurs then $(-1,1)$ and $(1,1)$ are wet, even if $G_{-1,1}$ and $G_{1,1}$ do not ocur.

Combining the definitions above we see that if $(y,n) \in W_n$ then $I_{y,n}$ is happy. We let
$$
{\cal C}_{x} = \{ (y,n) : (x,0) \to (y,n) \}
$$
be the {\bf cluster containing} $(x,0)$. Note that by our convention we always have $(x,0) \in {\cal C}_x$.

The event $G_{0,0}$ is not independent of $G_{0,2}$ or $G_{0,-2}$ because the associated boxes ${\cal R}_{m,n}$ overlap, but it is independent of all the other $G_{i,j}$. We say the 0,1 valued random variables $\eta(m,n)$ are {\bf M--dependent with density at least} $1-\gamma$ if whenever $(m_i,n_i)$, $1\le i \le k$ have $(|m_i-m_j| + |n_i-n_j|)/2  >M$ for all $i<j$ then
$$
P( \eta(m_i,n_i) = 0 \hbox{ for } 1\le i \le k ) \le \gamma^k
$$
Classical $M$-dependence would require that if the $(m_i,n_i)$ are all separated by at least $M$ then they are are independent, but in our definition we only consider the probability we need to control to prove the next two results

In the situations of interest here the variables are 1-dependent, but it is no harder to state the following results in general. The next two results are Theorems 4.1 and 4.2 from Durrett (1995).

\begin{theorem} \label{Block1}
Let $\gamma_M= 6^{-4(2M+1)^2}$. If $\gamma \le \gamma_M$ then
$P( |{\cal C}_0| < \infty) \le 55 \gamma^{1/(2M+1)^2} \le 1/20.$
\end{theorem}

\mn
As noted in Section \ref{sec:DN94} the pair and triplet creation model are attractive so if we start them from all 1's then the limit as $t\to\infty$ exists and is a translation invariant stationary distribution. When we start with all 1's then all of intervals $I_{0,2k}$ are happy. Let $W^1_{2n}$ be the set of wets sties at time $2n$ when all sites are wet at time 0.

\begin{theorem} \label{Block2}
If $\gamma \le \gamma_M$ then
$\liminf_{n\to\infty} P( 0 \in W^1_{2n} ) \ge 1 - 55 \gamma^{1/(2M+1)^2} \ge 19/20.$
\end{theorem}

We are now ready to prove Theorem 2 which states that as $\ep\to 0$, $\lambda_c(\ep) \to \lambda_1 = \inf\{ \lambda : r(\lambda)>0\}$ and that if $\lambda > \lambda_1$ then $P(\xi^1_\infty=1) \to \rho_1(\lambda)$. There are two parts to the argument.

\bn
{\bf c. Survival results} 

\mn
Here we describe the proof of $\limsup \lambda_c(\ep) \le \lambda_1$ and the limiting behavior of $P(\xi^1_\infty=1)$ for $\lambda>\lambda_1$. As mentioned earlier the happy configurations are those in which the density in $I_{m,n}$ is larger than $\rho_1+\delta$ in the sense of \eqref{densdef}. Using the balck construction with Theorem \ref{Block1} it is easy to conclude that if $\lambda> \lambda_1$ then for small $\ep$ there is a stationary distribution which gives us the first result.

To get the necessary lower bound on the density of 1's requires a little more work. The PDE result in Theorem \ref{PDEc2posr} allows us to show that if $\eta >0$ then with high probability the density in $[(2m-3)L,(2m+3)L]$, which contains $I_{m+1,n+1}$ and $I_{m-1,n+1}$,  is larger than $\rho_0-\eta$. Combining this with Theorem \ref{Block2} we can conclude that the limit starting from all 1's, which is translation invariant, has density at least  $\rho_0-\eta$.

Now that we have the desired lower bound on the fraction of occupied sites the rest is easy because the mean-field PDE allows us to conclude that the density is $\le \rho_0(\lambda)+\eta$. See Lemma 3.1 in DN(1994) for a proof. The reader can also find a general result in Theorem 1.4 of CDP(2013) which implies the result whose proof we just sketched.

\bn
{\bf d. Extinction results}

\mn
It remains to prove that if $\lambda < \lambda_1$ then the process dies out for small $\ep$. There are two steps: 

\medskip
(i) achieving small density. See Section 4 in DN(1994) or Chapter 4 in CDP(2013).

\medskip
(ii) killing off the particles. See Section 5 in DN(1994) or Chapter 7 in CDP(2013).

\mn
For (i) we define the happy configuration to be density at most $\rho_1-\delta$ in $I_{m,n}$. Using Theorem \ref{PDEc2negr} we can conclude that if $I_{m,n}$ is happy then there is a $\kappa >0$ so that with high probability the density in $[(2m-3)L,(2m+3)L]$ which contains $I_{m+1,n+1}$ and $I_{m-1,n+1}$  is smaller than $\ep^\kappa$. To achieve this upper bound, we have to run the process for time $a\log(1/\ep)$ where $a$ is small.

To prove (ii), we have to use estimates of the movement of random walks to prove that the particles die out. To do this, it is sufficent to consider the pair creation process which has an easier time to give births.

\clearp

\section{Proof of Theorem \ref{hydroLLN}} \label{sec:PfTh4} 

To do the proof it is useful to begin by using Greek letters for the parameters and then assigning values later. Let $\delta_\ep = \lceil \ep^{-\alpha} \rceil \ep \approx \ep^{1-\alpha}$,  let $n=\delta_\ep/\ep =    \lceil \ep^{-\alpha} \rceil$ be the number of sites in each block,  and let $m_\ep = \min\{ m: m \delta_\ep > \ep^{-\gamma} \}$ be (one-half) the number of blocks in $[-\ep^{-\gamma}, \ep^{-\gamma})$. Note that $m \le \ep^{\alpha-1-\gamma}$ and the total number of sites is $2\ep^{-1-\gamma}$.

\begin{proof}
The first step is to observe that by part f of Section \ref{sec:PfTh1} we can suppose that the initial condition is a product measure with $P(\xi_0(x)=1) = v(x)$ where $v$ is Lipschitz continuous on $\RR$.To estimate $|D^\ep(x,\xi_t) - u(t,x)|$ we will compute the 4th moment of the sum that defines $D^\ep(x,\xi_t)$, see \eqref{densdef0}, and use Markov's inequality. By definition  $E\xi_t(y)=u^\ep(t,y)$. Let 
\begin{align*}
X_i & = \xi^\ep_t(y_i) - u^{\ep}(t,y_i) \\
W_i & = \xi^Y_t(y_i) - \mu^\ep(t,y_i)
\end{align*}
where $\xi^Y_t(y_i)$ is the value the dual computes for $y_i$ when we replace the trajectories generated by stirring with independent random walks, and we let $\mu^\ep(t,y_i) = E\xi^Y_t(y_i)$.

Suppose that the sites in the block under consideration are indexed by $1, \ldots n$.
\beq
E\left(n^{-1}  \sum_{i=1}^n X_i\right)^{\sqz 4}  = n^{-4} \sum_S EX_S
\label{4sum}
\eeq
where the sum is over $S \in \{1,\ldots n\}^4$ and 
$$
X_S = X_{S(1)}X_{S(2)}X_{S(3)}X_{S(4)}.
$$ 
To bound the fourth moment we will replace $X_S$ by $W_S$. It is impossible to simultaneously approximate all of the duals starting from sites in one block, so to define the approximating independent random walk paths we only look at the sites in $S$. Since we are only considering four particles when we approximate by independent random walks it follows from \eqref{stirvind} combined with the Lipschitz continuity of $v(x)$ that
$$
|u^\ep(t,y_i) - \mu^\ep(t,y_i)| \le C_1 \ep^{0.45}. 
$$
A simple inequality for products of real numbers with absolute value  $<1$ implies that
$$
E|X_S - W_S| \le \sum_{i=1}^4 E|X_{S(i)}-W_{S(i)}| \le C_2\ep^{0.45},
$$
so using the triangle inequality we have
\beq
\left| n^{-4} E\sum_S X_S -  n^{-4} E\sum_S W_S \right| 
\le n^{-4} \sum_S E|X_S-W_S| \le C_2\ep^{0.45}.
\label{XvsW}
\eeq
In $ \sum_S W_S $ there are $O(n^2)$ terms of the form $W_i^2W^2_j$, or $W^4_i$ (the latter is a special care of the former when $i=j$). The remaining terms are $W_i^3W_j$, $W_i^2W_jW_k$, or $W_iW_jW_kW_\ell$. Since the $W_i$ are independent, $\le 1$ and have mean 0
$$
0 \le n^{-4} E \sum_S W_S  \le C_3 n^{-2} \le C_3 \ep^{2\alpha},
$$
where the nonnegativity follows from \eqref{4sum}. If $2\alpha>0.45$ then using 
\eqref{XvsW} we have for small $\ep$ that 
$$
0 \le n^{-4} E \sum_S X_S  \le 2C_2 \ep^{0.45}
$$
and hence $P\left( n^{-4} E \sum_S X_S  \ge \ep^\beta \right) \le C_3 \ep^{0.45-\beta}$. Rewriting this using \eqref{4sum} it follows that 
$$
P\left( \left(n^{-1}  \sum_S X_S \right)^4  \ge \ep^\beta \right)
 \le C_3 \ep^{0.45-\beta}
$$
If we set $\alpha = 0.9$, $\beta=0.04$ and $\gamma=0.3$, then we get the inequality in Theorem \ref{hydroLLN}.
\end{proof}

\clearp

\clearp

\section*{References}

\mn
Biskup, M., and Chayes, L (2003)
Rigorous analysis of discontinuous phase transitions via mean-filed bounds.
{\it Commun. Math. Phys.} 238, 53-93

\mn
Cardozo, G.O., and Fontanari, J.F. (2006)
Revisiting the nonequilibirum phase transition of the triplet-creation model.
{\it Eur. Phys. J., B.} 51, 555--

\mn
 Chatterjee, S., and Durrett, D. (2013) 
A first order phase transition in the threshold $\theta \ge 2$ contact process on random $r$-regular graphs and r-trees. 
{\it Stoch. Proc. Appl.} 123), 561--578

\mn
Cox, J.T., Durrett, R. and Perkins, E.A. (2013)
Voter model perturbations and reaction diffusion equations. 
Asterique. Volume 349, (113 pages). Also available at arXiv:1103.1676

\mn
Dickman, R., and Tom\'e, T. (1991)
First order phase transitions in a one-dimensional nonequilbrium model.
{\it Phys. Rev. A.} 44, 4833--4838

\mn
Durrett, R. (1995)
{\it Ten Lectures on Particle Systems.}
 Pages 97--201 in St. Flour Lecture Notes. 
Lecture Notes in Math 1608. Springer-Verlag, New York

\mn
Durrett, R. (2027)
{\it Interacting Particle Systems: Ideas, Techniques, Applications.} \hbr
\url{https://sites.math.duke.edu/~rtd/PASTA/PASTAcon.html}

\mn
Durrett, R., and Neuhauser, C. (1994)
Particle systems and rezction diffusion equations.
{\it Ann. Probab.} 22, 289--333

\mn
Fife, P.C., and McLeod, J.B. (1977)
The approach of solutions of nonlinear differential equations
to traveling front solutions.
{\it Arch. Rat. Mech. Anal.} 65, 335--361

\mn
Fiore, C.E., and de Oliveira, M.J. (2004)
Phase transition in conservative diffusive contact processes.
{\it Phys. Rev. E.} 70, paper 046331

\mn
Griffeath, D. (1979)
{\it Additive and Cancellative Interacting Particle Systems.}
Springer Lecture Notes in Math 724

\mn
Hinrichsen, H. (2000a)
First-order transitions in fluctuating 1+1 dimensional nonequilibrium systems.
arXiv:cond-math/00006212

\mn
Hinrichsen, H. (2000b)
Non-equilibirum critical phenomena and phase transition into absorbing states.
{Advances in Physics.} 49, 815--898

\mn
Huang, X.,  and Durrett, R.  (2021)
Motion by mean curvature in interacting particle systems. 
{\it Probab. Theory Rel. Fields.} 181, 489-532

\mn
Kipnis, C., Olla, S., and Varadhan, S.R.S. (1989)
Hydrodynamics and large deviations for simple exclusion processes.
{it Comm. Pure Appl. Math.} 42, 115--137

\mn
Liggett, T.M. (1985)
{\it Interacting Particle Systems.}
Springer-Verlag, New York

\mn
Liggett, T.M. (1999)
{\it Stochastic Intereacting Systems: Contact, Voter and Exclusion Processes.}
Springer-Verlag, New York

\mn
Marro, J.,  and Dickman, R. (1999)
{\it  Nonequilibrium Phase Transitions in Lattice Models}
Cambridge University Press

\mn
\'Odor, G. (2003)
Phase transitions in triplet and quadruplet reaction-diffusion models.
{\it Phys. Rev. E.} 67, paper 056114

\mn
\'Odor, G., and Dickman, R. (2009)
On the absorbing-state phase transition in the one-diemnsional triplet creating model.
{\it J. Stat. Mech.: Theory and Experiment} 

\mn
Park, S-C (2009)
Absence of discontinuous transitions in the one-dimensional triplet creation model.
{\it Phys. Rev. E} 80, paper 061103

\end{document}